\documentclass[11pt]{amsart}

\usepackage{amsmath}
\usepackage{amssymb}
\usepackage{amsthm}
\usepackage{enumerate}
\usepackage{amsbsy}
\usepackage{amsfonts}

\usepackage[bitstream-charter]{mathdesign}
\usepackage[T1]{fontenc}

\usepackage{amsthm}
\newtheorem{thm}{Theorem}
\newtheorem{lem}[thm]{Lemma}

\newtheorem{rem}{Remark}

\newcommand{\Set}[1]{{\left\{#1\right\}}}
\newcommand{\abs}[1]{\left\vert#1\right\vert}
\newcommand{\Abs}[1]{\Big\vert#1\Big\vert}
\newcommand{\snoi}{\sum_{n=1}^{\infty}}
\newcommand{\intzi}{\int_0^\infty}
\newcommand{\intzo}{\int_0^1}
\newcommand{\intoi}{\int_1^\infty}

\begin{document}
\baselineskip=18pt

\title[Sampling of the Riemann zeta function]{An estimate of the second moment of a sampling of the Riemann zeta function on the critical line}

\author{Sihun Jo \& Minsuk Yang}

\address{Sihun Jo: School of Mathematics, Korea Institute for Advanced Study, 85 Hoegiro Dongdaemungu, Seoul, Republic of Korea}
\address{Minsuk Yang: School of Mathematics, Korea Institute for Advanced Study, 85 Hoegiro Dongdaemungu, Seoul, Republic of Korea}
\email{yangm@kias.re.kr}

\begin{abstract}
We investigate the second moment of a random sampling $\zeta(1/2+iX_t)$ of the Riemann zeta function on the critical line.
Our main result states that if $X_t$ is an increasing random sampling with gamma distribution, then for all sufficiently large $t$,
\[\mathbb{E} |\zeta(1/2+iX_t)|^2 = \log t + O(\sqrt{\log t}\log\log t).\]
\\

\end{abstract}

\maketitle

\section{Introduction}
\label{S1}

This paper is concerned with the behaviour of the Riemann zeta function $\zeta(s)$ along the critical strip $s=1/2+it$ by modelling the variable $t$ with a random sampling.
As is well known, the Riemann zeta function $\zeta(s)$ is defined as an analytic continuation of the function initially defined for all complex numbers $s=\sigma+it$ with real part greater than 1 by the absolutely convergent series
\[\zeta(s) = \snoi n^{-s}.\]

Lifshits and Weber \cite{MR2472167} studied the behaviour of the Riemann zeta function $\zeta(1/2+it)$, when $t$ is sampled by the Cauchy random walk.
They used Cauchy distribution because the necessary moment expressions for Cauchy distribution are by far more explicit than in other cases.
They remarked that they believe that the results similar to theirs are valid for sampling with a large class of random walks with discrete or continuous steps.

Here is our main result.

\begin{thm}
\label{S3:L3}
Let $X_t$ denote the gamma process with parameters $a=b=1$.
Then for all sufficiently large $t$,
\[\mathbb{E} |\zeta(1/2+iX_t)|^2 = \log t + O(\sqrt{\log t}\log\log t).\]
\end{thm}

The gamma process has two positive parameters $a$ and $b$.
We set $a=b=1$ for notational convenience.
We note that the gamma process is increasing, its average value is $t$, and its variance is $t$.
So, we use it to describe the situation how $\zeta(1/2+it)$ behaves as $t$ tends to infinity.
In the next section, we recall the definition and basic properties of the gamma process.
We will extensively use the Landau notation $f=O(g)$, which means that $|f(x)|\le Cg(x)$ for some unspecified constant $C$. 
We also use the Vinogradov notation $f \ll g$; it is equivalent to $f=O(g)$.

We now give a few remarks.

\begin{rem}
Theorem \ref{S3:L3} is a probabilistic analog of the famous result obtained by Hardy and Littlewood; as $T\to\infty$,
\[\frac{1}{T} \int_0^T |\zeta(1/2+it)|^2 dt = \log T + O(1).\]
\end{rem}

\begin{rem}
Jutila \cite{MR705532} obtained the following estimate for value distribution of the Riemann zeta function
\[\frac{1}{T} \mathrm{meas}(M_T(V)) \ll \exp\left(-\frac{\log^2 V}{\log\log T}\left(1+O\left(\frac{\log V}{\log\log T}\right)\right)\right),\]
where $T\ge2$, $1\le V\le\log T$, and $M_T(V) = \Set{0\le t\le T: |\zeta(1/2+it)| \ge V}$.
\end{rem}

\begin{rem}
From Chebyshev's inequality, it is easy to see that the random sampling $\frac{|\zeta(1/2+iX_t)|}{\log t}$ converges to zero in probability.
It is also a probabilistic analog of Jutila's result with $V=\log T$, although his result is much stronger.
\end{rem}

Let us explain our method of proof.
We begin by analytically extending the zeta function with suitable form, and then investigate the moment of the sampling $\zeta(\sigma+iX_t)$.
Taking expectation is equivalent to considering the Fourier transform of probability measure of gamma process.
The resulting equation is a type of oscillatory integrals.
It is easy to estimate the first moment by applying repeated integration by parts.
However, it is not so easy to estimate the second moment.
The key idea in our strategy is divide and conquer algorithm based on recursion and iteration.
We decompose sums and integrals into several pieces and then use a refined version of van der Corput's method to estimate each terms.
The argument is technically elementary, but delicate.

\section{Preliminaries}
\label{S2}

\subsection{Gamma process}

The gamma process $X_t^{a,b}$ plays a role as a natural continuous time analogue of independent and identically distributed sequence of positive increasing random variables.
It is a pure-jump increasing L\'evy process with independent gamma distributed increments with two positive parameters $a$ and $b$.
The law of $X_t^{a,b}$ is given by
\[d\mathcal{P}_{X_t^{a,b}}(x) = 1_{(0,\infty)}(x) \frac{b^{at}}{\Gamma(at)}x^{at-1}e^{-bx} dx.\]
The parameter $a$ controls the rate of jump arrivals and the scaling parameter $b$ inversely controls the jump size.
From now, we fix the parameter $a=b=1$ for convenience and denote by $X_t$ the corresponding gamma process.
A calculation shows that its average value is $\mathbb{E}(X_t) = \intzi x d\mathcal{P}_{X_t}(x) = t$, and its variance is $\mathbb{E}|X_t|^2-|\mathbb{E}X_t|^2 = t$. 
The characteristic function of $X_t$ is defined by the Fourier transform of the law $\mathcal{P}_{X_t}$, that is,
\begin{equation}\label{S3:E4}
\mathbb{E}(e^{iu X_t}) = (1-iu)^{-t}.
\end{equation}

\subsection{Van der Corput's lemmas}

The proof of the following lemmas can be found, for example, in \cite{MR2061214}.

\begin{lem}
\label{S2:L1}
Let $f(x)$ be a real-valued function with $f''(x)>0$ on the interval $[a,b]$ and let $g$ be any smooth function $[a,b]$.
Then
\begin{align*}
\sum_{a<n<b} g(n) \exp(2\pi if(n)) 
&=\sum_{\alpha-\epsilon<m<\beta+\epsilon}\int_a^b g(x) \exp(2\pi i (f(x)-mx)) dx \\
&\quad+O(G(\epsilon^{-1}+\log (\beta-\alpha+2))),
\end{align*}
where
$\alpha,\beta$ and $\epsilon$ are any numbers with $\alpha\leq f'(a)\leq f'(b)\leq \beta$ and $0<\epsilon\leq 1$, $G=|g(b)|+\int_a^b|g'(y)|dy$, and the implied constant is absolute.
\end{lem}

\begin{lem}
\label{S2:L2}
Let $f(x)$ be a real function with $|f'(x)|\leq 1-\theta$ and $f''(x)\neq 0$ on $[a,b]$.
Then
\[\sum_{a<n<b} g(n) \exp(2\pi i f(n)) = \int_a^b g(x) \exp(2\pi if(x)) dx + O(G\theta^{-1}),\]
where $G=|g(b)|+\int_a^b|g'(y)|dy$ and
the implied constant is absolute.
\end{lem}

The following lemma is a modification of Lemma \ref{S2:L1}.

\begin{lem}
\label{S2:L3}
Let $f(x)$ be a real-valued function with $f''(x)>0$ on the
interval $[a,b]$ and let $g$ be any smooth function $[a,b]$, where
$b-a>2$. Then
\begin{align*}
\sum_{a<n<b} g(n) \exp(2\pi if(n)) 
&=\sum_{\alpha-\eta<m<\beta+\eta}\int_a^b g(x) \exp(2\pi i (f(x)-mx)) dx \\
&+ O\left(G_1\bigg(\eta^{-1}+\log \Big(1+\frac{\beta-\alpha}{\eta}\Big)\bigg)+G_2(\beta-\alpha+\eta)\right),
\end{align*}
where
$\alpha= f'(a)\leq f'(b)= \beta$, $\eta > 1 $, $G_1=|g(b)|+\int_a^b|g'(y)|dy$, $G_2=\max\{|g(x)| : x\in [a,a+1]\cup [b-1,b]\}$ and the implied constant is
absolute.
\end{lem}

\section{Proof of Theorem}
\label{S3}

We begin by proving an analytic continuation of the Riemann zeta function.

\begin{lem}
\label{S3:L1}
Let $\{u\}$ denote the fractional part of $u$.
For $0<\sigma<1$, 
\[\zeta(s) = 1 - \intzo u^{-s} du + \intoi \{u\} \frac{d}{du} u^{-s}du.\]
\end{lem}

\begin{proof}
For $\sigma>1$,
\[\zeta(s) = 1+\intoi u^{-s}d[u] = 1+\intoi u^{-s}du-\intoi u^{-s}d\{u\}.\]
An integration by parts yields an analytic continuation of $\zeta(s)$ into the half-plane $\sigma>0$, that is,
\[\zeta(s) = 1-\frac{1}{1-s}+\intoi \{u\} \frac{d}{du} u^{-s}du.\]
Since $\sigma>0$, we have $\frac{1}{1-s} = \intzo u^{-s} du$.
\end{proof}

We estimate the mean value of the sampling of the Riemann zeta function.

\begin{lem}
\label{S3:L2}
For all $N\in\mathbb{N}$ and all sufficiently large $t$, 
\[\mathbb{E} \zeta(1/2+iX_t) = 1 + O(t^{-N}).\]
\end{lem}

\begin{proof}
By Lemma \ref{S3:L1} and the fact $\mathbb{E} u^{-iX_t}=(1+i\log u)^{-t}$, 
\[\mathbb{E} \zeta(1/2+iX_t) = 1 - \intzo u^{-1/2} (1+i\log u)^{-t} du + \intoi \{u\} \frac{d}{du} \left(u^{-1/2} (1+i\log u)^{-t}\right) du.\]
The last integral becomes
\begin{align*}
&\intoi \{u\} \frac{d}{du} \left(u^{-1/2} (1+i\log u)^{-t}\right) du \\
&= \int_1^2 (u-1) \frac{d}{du} \left(u^{-1/2} (1+i\log u)^{-t}\right) du + O(e^{-ct}) \\
&=-\int_1^2u^{-1/2}(1+i\log u)^{-t} du + O(e^{-ct})
\end{align*}
for some positive real number $c$.
Therefore
\[\mathbb{E} \zeta(1/2+iX_t) = 1-\int_0^2 u^{-1/2}(1+i\log u)^{-t} du + O(e^{-ct}).\]
Since
\[(1+i\log u)^{-t} = \frac{iu}{t-1} \frac{d}{du}(1+i\log u)^{-t+1},\]
an $N$-fold integration by parts gives
\begin{align*}
&\int_0^2 u^{-1/2}(1+i\log u)^{-t} du \\
&= \frac{(i/2)^N}{(t-1)(t-2)\cdots(t-N)} \int_0^2 u^{-1/2}(1+i\log u)^{-t+N}du + O(e^{-ct}).
\end{align*}
The result follows.
\end{proof}

We now ready to prove our main theorem.

\textbf{[Step 1]}
We claim that 
\begin{equation}
\mathbb{E} |\zeta(1/2+iX_t)|^2 = t(t+1) \iint_R \{v\} v^{-3/2} \{u\} u^{-3/2} \left(1+i\log\frac{u}{v}\right)^{-t-2} dudv + O(1),
\end{equation}
where 
\begin{equation}
R(t)=\Set{(u,v) : 1<u<t^4, 1<v<t^4, \Abs{\log\frac{u}{v}} < 2\sqrt{\frac{\log t}{t}}}.
\end{equation}

To see this, using Lemma \ref{S3:L2} we have 
\[\mathbb{E}|\zeta(1/2+iX_t)|^2 = \mathbb{E}|\zeta(1/2+iX_t)-1|^2+O(1).\]
Using Lemma \ref{S3:L1} we expand 
\begin{align*}
\mathbb{E}|\zeta(1/2+iX_t)-1|^2
&=\mathbb{E}\left|\intzo u^{-1/2-iX_t} du - \intoi \{u\} \frac{d}{du} u^{-1/2-iX_t}du\right|^2 \\
&=A_1(t)-2\Re A_2(t)+A_3(t),
\end{align*}
where
\begin{align*}
A_1(t)&=\intzo\intzo u^{-1/2} v^{-1/2} \mathbb{E} (u/v)^{-iX_t} dudv, \\
A_2(t)&=\intoi \{v\}\frac{d}{dv}\left(v^{-1/2}\intzo u^{-1/2} \mathbb{E}(u/v)^{-iX_t} du\right)dv, \\
A_3(t)&=\intoi \{v\} \frac{d}{dv} \left(v^{-1/2} \intoi \{u\} \frac{d}{du} u^{-1/2} \mathbb{E} (u/v)^{-iX_t} du\right) dv.
\end{align*}
It is easy to see that $A_1(t)+A_2(t) = O(1)$. 
Indeed, 
\[\Abs{\mathbb{E} (u/v)^{-iX_t}} = \Abs{\left(1+i\log\frac{u}{v}\right)^{-t}} \le 1,\]
and a change of variables shows 
\begin{align*}
A_2(t)
&=\intoi \{v\}\frac{d}{dv}\left(v^{-1/2}\intzo u^{-1/2}\left(1+i\log\frac{u}{v}\right)^{-t}du\right)dv \\
&=\intoi \{v\}\frac{d}{dv}\left(\int_0^{1/v} u^{-1/2}(1+i\log u)^{-t}du\right)dv \\
&=-\intoi\{v\} v^{-3/2}(1-i\log v)^{-t}dv \\
&= O(1).
\end{align*}
A direct calculation shows 
\begin{align*}
A_3(t)
&=\intoi \{v\} \frac{d}{dv} \left(v^{-1/2} \intoi \{u\} \frac{d}{du} \left(u^{-1/2} \left(1+i\log\frac{u}{v}\right)^{-t}\right) du\right) dv \\
&=t(t+1) \intoi \intoi \{v\} v^{-3/2} \{u\} u^{-3/2} \left(1+i\log\frac{u}{v}\right)^{-t-2} du dv+O(1).
\end{align*}
We observe that 
\[\int_{t^4}^\infty x^{-3/2} dx \ll t^{-2}\]
and that if $\Abs{\log\frac{u}{v}} \ge 2\sqrt{\frac{\log t}{t}}$ then
\[\Abs{\left(1+i\log\frac{u}{v}\right)^{-t-2}} \le \exp\left(-\frac{t}{2}\log\left(1+4\frac{\log t}{t}\right)\right) \ll t^{-2}.\]
So we can reduce the domain of integration and this proves the claim.

\textbf{[Step 2]}
We claim that 
\begin{equation}
t(t+1) \iint_R \{v\} v^{-3/2} \{u\} u^{-3/2} \left(1+i\log\frac{u}{v}\right)^{-t-2} dudv = \sum_{(m,n)\in R(t)} \left(\widetilde{F}_{m,n}(t) - \widetilde{G}_{m,n}(t)\right) + O(1),
\end{equation}
where
\begin{equation}
\widetilde{F}_{m,n}(t) = \frac{1}{\sqrt{n+1}} \frac{1}{\sqrt{m+1}} \exp\left(-t\left(i\log{\frac{m+1}{n+1}} + \frac{1}{2}\log^2{\frac{m+1}{n+1}}\right)\right)
\end{equation}
and 
\begin{equation}
\widetilde{G}_{m,n}(t) = \frac{1}{\sqrt{m+1}} \int_n^{n+1} v^{-1/2} \exp\left(-t\left(i\log{\frac{m+1}{v}}+\frac{1}{2}\log^2{\frac{m+1}{v}}\right)\right) dv.
\end{equation}

To see this, we write 
\[t(t+1) \iint_R \{v\} v^{-3/2} \{u\} u^{-3/2} \left(1+i\log\frac{u}{v}\right)^{-t-2} dudv =  t(t+1) \sum_{(m,n)\in R(t)} C_{m,n}(t)+O(1)\]
where $m,n \in \mathbb{N}$ and 
\[C_{m,n}(t) = \int_n^{n+1} \frac{v-n}{v^{3/2}} \int_m^{m+1} \frac{u-m}{u^{3/2}} \left(1+i\log\frac{u}{v}\right)^{-t-2} dudv.\]
Applying an integration by parts to the inner integral yields
\begin{align*}
C_{m,n}(t)
&= \frac{i}{t+1} \frac{1}{\sqrt{m+1}} \int_n^{n+1} \frac{v-n}{v^{3/2}} \left(1+i\log\frac{m+1}{v}\right)^{-t-1} dv \\
&\quad -\frac{i}{t+1} \int_n^{n+1} \frac{v-n}{v^{3/2}} \int_m^{m+1} u^{-1/2} \left(1+i\log\frac{u}{v}\right)^{-t-1} dudv \\
&\quad +\frac{i/2}{t+1} \int_n^{n+1} \frac{v-n}{v^{3/2}} \int_m^{m+1} \frac{u-m}{u^{3/2}} \left(1+i\log\frac{u}{v}\right)^{-t-1} dudv.
\end{align*}
If we denote 
\begin{align*}
D_{m,n}(t) &= \frac{1}{\sqrt{m+1}} \int_n^{n+1} \frac{v-n}{v^{3/2}} \left(1+i\log\frac{m+1}{v}\right)^{-t-1} dv \\
E_{m,n}(t) &= \int_n^{n+1} \frac{v-n}{v^{3/2}} \int_m^{m+1} u^{-1/2} \left(1+i\log\frac{u}{v}\right)^{-t-1} dudv,
\end{align*}
then we rewrite the above identity as
\[C_{m,n}(t) = \frac{i}{t+1} D_{m,n}(t) - \frac{i}{t+1} E_{m,n}(t) + \frac{i/2}{t+1} C_{m,n}(t-1).\]
Because we can iterate this relation, we have 
\begin{align*}
&t(t+1) \iint_R \{v\} v^{-3/2} \{u\} u^{-3/2} \left(1+i\log\frac{u}{v}\right)^{-t-2} dudv \\
&= it \sum_{(m,n)\in R(t)} \left(D_{m,n}(t) - E_{m,n}(t)\right) + O(1).
\end{align*}

Now, we show that $it \sum_{(m,n)\in R(t)} E_{m,n}(t)$ is an error term.
By making the change of variables we have 
\begin{align*}
\sum_{(m,n)\in R(t)} E_{m,n}(t)
&\ll \sum_{1\le n < t^4} \int_n^{n+1} \frac{v-n}{v^{3/2}} \int_{\abs{\log\frac{u}{v}} < 2\sqrt{\frac{\log t}{t}}} u^{-1/2} \left(1+i\log\frac{u}{v}\right)^{-t-1} dudv \\
&= \sum_{1\le n < t^4} \int_n^{n+1} \frac{v-n}{v} \int_{\exp(-2\sqrt{\frac{\log t}{t}})}^{\exp(2\sqrt{\frac{\log t}{t}})} u^{-1/2} \left(1+i\log u\right)^{-t-1} dudv.
\end{align*}
Using the identity
\[(1+i\log u)^{-t-1} = \frac{iu}{t} \frac{d}{du}(1+i\log u)^{-t},\]
we perform an integration by parts to obtain that 
\[\int_{\exp(-2\sqrt{\frac{\log t}{t}})}^{\exp(2\sqrt{\frac{\log t}{t}})} u^{-1/2} \left(1+i\log u\right)^{-t-1} du = O(t^{-3}).\]
So we have 
\[\sum_{(m,n)\in R(t)} E_{m,n}(t) \ll \int_1^{t^4} \frac{1}{v} t^{-3} dv \ll t^{-3} \log t.\]

Similarly, an integration by parts yields
\[D_{m,n}(t) = -\frac{i}{t} F_{m,n}(t) + \frac{i}{t} G_{m,n}(t) - \frac{i/2}{t} D_{m,n}(t-1),\]
where
\begin{align*}
F_{m,n}(t) &= \frac{1}{\sqrt{n+1}} \frac{1}{\sqrt{m+1}} \left(1+i\log\frac{m+1}{n+1}\right)^{-t} \\
G_{m,n}(t) &= \frac{1}{\sqrt{m+1}} \int_n^{n+1} v^{-1/2} \left(1+i\log\frac{m+1}{v}\right)^{-t} dv.
\end{align*}
From this identity, we then obtain
\[A_3(t) = \sum_{(m,n)\in R} \left(F_{m,n}(t) - G_{m,n}(t)\right) + O(1).\]
Finally, in the region $R(t)$, we can approximate the exponential function $(1+z)^{-t}$, that is, for $|z| < 2\sqrt{\frac{\log t}{t}}$,
\[(1+z)^{-t} = \exp(-t(z-z^2/2)) (1+O(t|z|^3)).\]
This proves the claim.

\textbf{[Step 3]}
We claim that 
\begin{equation}
\sum_{t<n<t^4} \sum_{m\in R(n,t)} \left(\widetilde{F}_{m,n}(t) - \widetilde{G}_{m,n}(t)\right) = O(1),
\end{equation}
where
\[R(n,t)=\Set{m\in\mathbb{N} : \Abs{\log\frac{m+1}{n+1}} < 2\sqrt{\frac{\log t}{t}}}.\]

We only prove the sum for $\widetilde{F}_{m,n}(t)$ because the case for $\widetilde{G}_{m,n}(t)$ is almost the same.
We have 
\[\sum_{t<n<t^4} \sum_{m\in R(n,t)} \widetilde{F}_{m,n}(t) \ll \sum_{t<n<t^4} \frac{1}{\sqrt{n}} \sum_{m\in R(n,t)} \frac{1}{\sqrt{m}} \exp\left(-\frac{t}{2}\log^2\frac{m}{n}\right) \exp\left(-ti\log\frac{m}{n}\right).\]
If we set $\eta=\log t$ in Lemma \ref{S2:L3}, then 
\begin{align*}
&\sum_{m\in R(n,t)} \frac{1}{\sqrt{m}} \exp\left(-\frac{t}{2}\log^2\frac{m}{n}\right) \exp\left(-ti\log\frac{m}{n}\right) \\
&=\sum_{-\log t<k<\log t} \int_{n\exp(-2\sqrt{\frac{\log t}{t}})}^{n\exp(2\sqrt{\frac{\log t}{t}})} x^{-1/2} \exp\left(-\frac{t}{2}\log^2\frac{x}{n}\right) \exp\left(-ti\log\frac{x}{n}-2\pi ikx\right) dx \\
&\quad+ O\left(\frac{1}{\sqrt{n}\log t} + \frac{\log t}{\sqrt{n}t^2}\right).
\end{align*}
It is easy to see that 
\[\sum_{t<n<t^4} \frac{1}{\sqrt{n}} \left(\frac{1}{\sqrt{n}\log t} + \frac{\log t}{\sqrt{n}t^2}\right) \ll 1.\]
In order to estimate the main term, we change variables to get
\begin{align*}
&\sum_{-\log t<k<\log t} \int_{n\exp(-2\sqrt{\frac{\log t}{t}})}^{n\exp(2\sqrt{\frac{\log t}{t}})} x^{-1/2} \exp\left(-\frac{t}{2}\log^2\frac{x}{n}\right) \exp\left(-ti\log\frac{x}{n}-2\pi ikx\right) dx \\
&= \sqrt{n} \sum_{-\log t<k<\log t} \int_{-2\sqrt{\frac{\log t}{t}}}^{2\sqrt{\frac{\log t}{t}}} \exp(x/2) \exp\left(-t(ix+x^2/2)-2\pi ikne^x\right) dx.
\end{align*}
Using the identity
\begin{align*}
&\exp\left(-t(ix+x^2/2)-2\pi ikne^x\right) \\
&= -\frac{1}{t(i+x)+2\pi ikne^x} \frac{d}{dx} \exp\left(-t(ix+x^2/2)-2\pi ikne^x\right),
\end{align*}
we can integrate by parts to obtain 
\begin{align*}
&\int_{-2\sqrt{\frac{\log t}{t}}}^{2\sqrt{\frac{\log t}{t}}} \exp(x/2) \exp\left(-t(ix+x^2/2)-2\pi ikne^x\right) dx \\
&=-\left[\frac{\exp(x/2)}{t(i+x)+2\pi ikne^x} \exp\left(-t(ix+x^2/2)-2\pi ikne^x\right)\right]_{-2\sqrt{\frac{\log t}{t}}}^{2\sqrt{\frac{\log t}{t}}} \\
&\quad+\int_{-2\sqrt{\frac{\log t}{t}}}^{2\sqrt{\frac{\log t}{t}}} \frac{d}{dx} \left(\frac{\exp(x/2)}{t(i+x)+2\pi ikne^x}\right) \exp\left(-t(ix+x^2/2)-2\pi ikne^x\right) dx.
\end{align*}
Notice that for all sufficiently large $t$, $2\sqrt{\frac{\log t}{t}}<\frac{1}{10}$.
If $k=0$, then $\frac{1}{t(i+x)}\ll \frac{1}{t}$.
If $k\neq 0$, then we have $|2\pi kne^x| \ge 2t$ and so we have $\frac{1}{t(i+x)+2\pi ikne^x}\ll \frac{1}{kn}$.
So the function $\frac{\exp(x/2)}{i+x+2\pi i kne^x/t}$ and its derivatives are bounded by a constant uniformly on the domain of integration.
Hence we can repeat the integration by parts so that the main term comes from the boundary values
\[\left[\frac{\exp(x/2)}{t(i+x)+2\pi ikne^x} \exp\left(-t(ix+x^2/2)-2\pi ikne^x\right)\right]_{-2\sqrt{\frac{\log t}{t}}}^{2\sqrt{\frac{\log t}{t}}} \le \min\Set{\frac{1}{t^3},\frac{1}{knt^2}}.\]
Therefore the main term becomes
\[\sum_{t<n<t^4} \frac{1}{\sqrt{n}} \sum_{0<k<\log t} \frac{1}{knt^2} \ll \frac{\log\log t}{t^2\sqrt{t}}.\]
This proves the claim.

\textbf{[Step 4]}
It is easy to see that 
\begin{equation}
\sum_{1\le n\le t} \left(\widetilde{F}_{n,n}(t) - \widetilde{G}_{n,n}(t)\right) = \log t+O(1).
\end{equation}
Indeed, we have 
\[\sum_{1\le n\le t} \widetilde{F}_{n,n}(t) = \sum_{1\le n\le t} \frac{1}{n+1}\]
and 
\begin{align*}
\sum_{1\le n\le t} \widetilde{G}_{n,n}(t)
&=\sum_{1\le n\le t} \frac{1}{\sqrt{n+1}} \int_n^{n+1} v^{-1/2} \exp\left(-t\left(i\log{\frac{n+1}{v}}+\frac{1}{2}\log^2{\frac{n+1}{v}}\right)\right) dv \\
&=\sum_{1\le n\le t} \int_1^{1+1/n} v^{-3/2} \exp\left(-t\left(i\log v+\frac{1}{2}\log^2v\right)\right) dv.
\end{align*}
An integration by parts shows that the last sum is bounded by a constant.

Because the summing on $\Set{m\in\mathbb{N} : -2\sqrt{\frac{\log t}{t}} < \log\frac{m+1}{n+1} < 0}$ has the same estimates, we need to estimate the sum
\[\sum_{1\le n\le t} \sum_{m\in R_+(n,t)} \left(\widetilde{F}_{m,n}(t) - \widetilde{G}_{m,n}(t)\right),\]
where 
\[R_+(n,t)=\Set{m\in\mathbb{N} : 0 < \log\frac{m+1}{n+1} < 2\sqrt{\frac{\log t}{t}}}.\]
Moreover, since $\widetilde{G}_{m,n}(t)$ is Riemann integrable, we may consider the sum
\[\sum_{1\le n\le t} \sum_{m\in R_+^\delta(n,t)} \frac{1}{\sqrt{m+1}} \frac{1}{\sqrt{n+\delta}} \exp\left(-t\left(i\log{\frac{m+1}{n+\delta }} + \frac{1}{2}\log^2{\frac{m+1}{n+\delta}}\right)\right),\]
where $0<\delta\le 1$ and
\[R_+^\delta(n,t)=\Set{m\in\mathbb{N} : 1 < \frac{m+1}{n+\delta} < 1+2\sqrt{\frac{\log t}{t}}}.\]
Finally, we observe that the Taylor expansion shows 
\begin{align*}
&\sum_{0 < k-\delta \le 2(n+\delta)\sqrt{\frac{\log t}{t}}} \frac{1}{\sqrt{n+k}} \exp\left(-t\left(i\log\left(1+\frac{k-\delta}{n+\delta}\right)+\log^2\left(1+\frac{k-\delta}{n+\delta}\right)\right)\right)\\
&=\sum_{0 < k-\delta \le 2(n+\delta)\sqrt{\frac{\log t}{t}}} \frac{1}{\sqrt{n+k}} \exp\left(-\frac{t(k-\delta)^2}{(n+\delta)^2}\right) \exp\left(-it\left(\frac{k-\delta}{n+\delta}-\frac{(k-\delta)^2}{2(n+\delta)^2}\right)\right) \big(1+O(t^{-\frac{1}{2}})\big).
\end{align*}

\textbf{[Step 5]}
We now consider the sum
\begin{equation}
\sum_{\frac{1}{2}\sqrt{\frac{t}{\log t}} < n+\delta < \sqrt{t\log t}} \frac{1}{\sqrt{n+\delta}} \sum_{0 < k-\delta \le 2(n+\delta)\sqrt{\frac{\log t}{t}}} \frac{1}{\sqrt{n+k}} \exp\left(-\frac{t(k-\delta)^2}{(n+\delta)^2}\right) \exp\left(-it\left(\frac{k-\delta}{n+\delta}-\frac{(k-\delta)^2}{2(n+\delta)^2}\right)\right). 
\end{equation}
We claim that the order of the above sum is $O(\sqrt{\log t} \log\log t)$.

To see this, we apply Lemma \ref{S2:L1} so that the inner oscillatory sum becomes
\begin{align*}
&\sum_{\alpha-1/2<r<\beta+1/2} \int_\delta^{\delta+2(n+\delta)\sqrt{\frac{\log t}{t}}} \frac{1}{\sqrt{n+x}} \exp\left(-\frac{t(x-\delta)^2}{(n+\delta)^2}\right) \exp\left(-it\left(\frac{x-\delta}{n+\delta}-\frac{(x-\delta)^2}{2(n+\delta)^2}\right)-2\pi irx\right) dx \\
&\quad + O\big(G\big(\log (\beta-\alpha+2)\big)\big),
\end{align*}
where $\alpha=-\frac{t}{2\pi(n+\delta)}$ and $\beta=\frac{t}{2\pi(n+\delta)} \left(2\sqrt{\frac{\log t}{t}}-1\right)$.
It is easy to see that the error term is dominated by
\[G\big(\log (\beta-\alpha+2)\big) \ll \frac{\log\log t}{\sqrt{n+\delta}},\]
and that the main term is dominated by
\begin{align*}
&\sum_{\alpha-1/2<r<\beta+1/2} \int_\delta^{\delta+2(n+\delta)\sqrt{\frac{\log t}{t}}} 
\frac{1}{\sqrt{n+x}} \exp\left(-\frac{t(x-\delta)^2}{(n+\delta)^2}\right) dx \\
&\ll \frac{1}{\sqrt{n+\delta}}\sum_{\alpha-1/2<r<\beta+1/2} \frac{n+\delta}{\sqrt{t}} \ll \frac{\sqrt{\log t}}{\sqrt{n+\delta}}
\end{align*}
since $\beta-\alpha \ll \frac{\sqrt{t\log t}}{n+\delta}$.
Therefore
\[\sum_{\frac{1}{2}\sqrt{\frac{t}{\log t}} < n+\delta < \sqrt{t\log t}} \frac{1}{\sqrt{n+\delta}} \frac{\sqrt{\log t}}{\sqrt{n+\delta}} \ll \sqrt{\log t}\log\log t.\]

\textbf{[Step 6]}
We now consider the sum
\begin{equation}
\sum_{\sqrt{t\log t} \le n+\delta<t} \sum_{0 < k-\delta \le 2(n+\delta)\sqrt{\frac{\log t}{t}}} \frac{1}{\sqrt{n+\delta}} \frac{1}{\sqrt{n+k}} \exp\left(-\frac{t(k-\delta)^2}{(n+\delta)^2}\right) \exp\left(-it\left(\frac{k-\delta}{n+\delta}-\frac{(k-\delta)^2}{2(n+\delta)^2}\right)\right).
\end{equation}
We claim that the order of the above sum is $O(\log\log t)$.

We denote 
\[f_k^t(x)=\frac{t}{2\pi} \left(\frac{(k-\delta)^2}{2(x+\delta)^2}-\frac{k-\delta}{x+\delta}\right)\]
and
\[g_k^t(x)=\frac{1}{\sqrt{x+\delta}} \frac{1}{\sqrt{x+k}} \exp\left(-\frac{t(k-\delta)^2}{(x+\delta)^2}\right).\]
We then change the order of summation to obtain that 
\begin{align*}
&\sum_{\sqrt{t\log t} \le n+\delta<t} \sum_{0 < k-\delta \le 2(n+\delta)\sqrt{\frac{\log t}{t}}} g_k^t(n) \exp(2\pi if_k^t(n)) \\
&=\sum_{0<k-\delta\le2\log t} \sum_{\sqrt{t\log t}<n+\delta<t} g_k^t(n) \exp(2\pi if_k^t(n)) \\
&\quad+\sum_{2\log t<k-\delta<2\sqrt{t\log t}} \sum_{\frac{1}{2}\sqrt{\frac{t}{\log t}}(k-\delta)<n+\delta<t} g_k^t(n) \exp(2\pi if_k^t(n)) + O(t^{-1/2}\sqrt{\log t}) \\
&=:S_1(t) + S_2(t) + O(t^{-1/2}\sqrt{\log t}).
\end{align*}

We shall show that $S_2(t)=O(\log \log t)$.
By Lemma \ref{S2:L2} 
\[\sum_{\frac{1}{2}\sqrt{\frac{t}{\log t}}(k-\delta)<n+\delta<t} g_k^t(n) \exp(2\pi if_k^t(n)) = \int_{\frac{1}{2}\sqrt{\frac{t}{\log t}}(k-\delta)-\delta}^{t-\delta} g_k^t(x) \exp(2\pi if_k^t(x)) dx + O(G),\]
where 
\begin{align*}
\alpha &= \frac{d}{dx} f_k^t\left(t-\delta\right) = \frac{k-\delta}{2\pi t}\left(1-\frac{k-\delta}{t}\right), \\
\beta &= \frac{d}{dx} f_k^t\left(\frac{1}{2}\sqrt{\frac{t}{\log t}}(k-\delta)-\delta\right) = \frac{2\log t}{\pi(k-\delta)} \left(1-2\sqrt{\frac{\log t}{t}}\right), \\
G&=\max\left(g\left(\frac{1}{2}\sqrt{\frac{t}{\log t}}(k-\delta)-\delta\right), g\left(\sqrt{t}(k-\delta)-\delta\right), g(t-\delta)\right) \ll \frac{1}{\sqrt{t}(k-\delta)},
\end{align*}
since we have $-1<\alpha-\epsilon<\beta+\epsilon<1$ for $2\log t<k-\delta<2\sqrt{t\log t}$.
By the change of variables
\begin{align*}
&\int_{\frac{1}{2}\sqrt{\frac{t}{\log t}}(k-\delta)-\delta}^{t-\delta} g_k^t(x) \exp(2\pi if_k^t(x)) dx \\
&\ll \int_{\frac{\sqrt{t}(k-\delta)}{3\sqrt{\log t}}-\delta}^{\frac{t}{\pi}-\delta}\frac{1}{x+\delta}\exp\left(-it\left(\frac{k-\delta}{x+\delta}-
\frac{(k-\delta)^2}{2(x+\delta)^2}\right)\right) \exp\left(-\frac{t(k-\delta)^2}{(x+\delta)^2}\right)dx\\
&=\int_{\frac{\pi}{t}}^{\frac{3\sqrt{\log t}}{\sqrt{t}(k-\delta)}} \frac{1}{x} \exp\left(-it\left((k-\delta)x-\frac{(k-\delta)^2x^2}{2}\right)\right) \exp\left(-t(k-\delta)^2x^2\right) dx\\
\end{align*}
Using 
\begin{align*}
&\exp\left(-it\left((k-\delta)x-\frac{(k-\delta)^2x^2}{2}\right)-t(k-\delta)^2x^2\right) \\
&= -\frac{1}{2t(k-\delta)^2x+it(k-\delta)(1-(k-\delta)x)} \frac{d}{dx} \exp\left(-it\left((k-\delta)x-\frac{(k-\delta)^2x^2}{2}\right)-t(k-\delta)^2x^2\right)
\end{align*}
we integrate by parts and then estimate the resulting terms.
We did this type of estimates many times so we omit the details and give the results.
If $k-\delta<\sqrt{\frac{t}{\log t}}$, then it is dominated by 
\[\frac{\exp(-\pi (k-\delta)i)}{\pi k}+O\left(\frac{1}{(k-\delta)\log t}\right)+O\left(\frac{1}{\sqrt{t\log t}}\right),\]
and if $k-\delta\ge\sqrt{\frac{t}{\log t}}$, then it is dominated by $O\left(\frac{1}{k-\delta}\right)+O\left(\frac{1}{\sqrt{t\log t}}\right)$. 
Therefore,
\begin{align*}
S_2(t) 
&\ll \sum_{2\log t<k-\delta<\sqrt{\frac{t}{\log t}}} \left(\frac{\exp(i\pi \delta)}{\pi}\frac{ (-1)^k }{k-\delta}+\frac{1}{(k-\delta)\log t}+\frac{1}{\sqrt{t\log t}}\right) \\
&\quad+ \sum_{\sqrt{\frac{t}{\log t}}<k-\delta<2\sqrt{t\log t}} \left(\frac{1}{k-\delta}+\frac{1}{\sqrt{t\log t}}\right)=O(\log \log t).
\end{align*}

By Lemma \ref{S2:L2} 
\[\sum_{\sqrt{t\log t}<n+\delta<t} g_k^t(n) \exp(2\pi if_k^t(n)) = \int_{\sqrt{t\log t}-\delta}^{t-\delta} g_k^t(x) \exp(2\pi if_k^t(x)) dx  + O(G),\]
where 
\begin{align*}
\alpha &= \frac{d}{dx} f_k^t\left(t-\delta\right) = \frac{k-\delta}{2\pi t}\left(1-\frac{k-\delta}{t}\right), \\
\beta &= \frac{d}{dx} f_k^t\left(\sqrt{t\log t}-\delta\right) = \frac{k-\delta}{2\pi\log t} \left(1-\frac{k-\delta}{\sqrt{t\log t}}\right), \\
G &= \max\left(g\left(\sqrt{t\log t}-\delta\right),g\left(\sqrt{t}(k-\delta)-\delta\right),g(t-\delta)\right) \ll \max\left(\frac{1}{\sqrt{t\log t}},\frac{1}{\sqrt{t}(k-\delta)}\right),
\end{align*}
since we have $-1<\alpha-\epsilon<\beta+\epsilon<1$ for $0< k-\delta\leq 3\log t$.
Thus, it suffices to estimate
\[\int_{\sqrt{t\log t}-\delta}^{\frac{t}{\pi}-\delta}\frac{1}{x+\delta}\exp\left(-it\left(\frac{k-\delta}{x+\delta}-\frac{(k-\delta)^2}{2(x+\delta)^2}\right)\right) \exp\left(-\frac{t(k-\delta)^2}{(x+\delta)^2}\right)dx.\]
Changing variables and then integrating by parts shows that the main term is dominated by 
\[\frac{\exp(-\pi (k-\delta)i)}{\pi (k-\delta)}+O\left(\frac{1}{(k-\delta)^2}\right)+O\left(\frac{\sqrt{\log t}}{\sqrt{t}(k-\delta)}\right).\]
Therefore,
\[S_1(t) \ll \sum_{0<k-\delta\le2\log t} \left(\frac{\exp(i\pi \delta)}{\pi}\frac{ (-1)^k}{k-\delta}+\frac{1}{(k-\delta)^2}+\frac{\sqrt{\log t}}{\sqrt{t}(k-\delta)}\right)=O(1).\]
This complete the proof of Theorem.

\section*{Acknowledgements}
\label{S4}

The authors wish to express their sincere gratitude to professor Haseo Ki for giving related references and for his helpful comments which greatly improve the quality of this manuscript.

\end{document}